\documentclass[12pt]{amsart}

\usepackage[utf8]{inputenc}
\usepackage{amsrefs, amsmath,  amssymb }
\usepackage[all,cmtip]{xy}
\usepackage{bbold}
\usepackage{bbm}
\usepackage{mathtools}

\date{}

\newcommand{\g}{{\mathfrak g}}
\newcommand{\n}{{\mathfrak n}}
\newcommand{\h}{{\mathfrak h}}
\newcommand{\gl}{{\mathfrak {gl}}}
\newcommand{\Z}{{\mathbb Z}}

\newcommand{\Q}{{\mathbb Q}}
\newcommand{\C}{{\mathbb C}}
\newcommand{\cf}{\mathcal F}
\newcommand{\cc}{\mathcal C}
\newcommand{\ck}{\mathbbm{k}}
\newcommand{\nn}{{(n)}}
\newcommand{\Rep}{\operatorname{Rep}}

\renewcommand{\deg}{\operatorname{deg}~}

\newcommand{\En}{\operatorname{End}}
\newcommand{\Hom}{\operatorname{Hom}}
\newcommand{\Ind}{\operatorname{Ind}}
\newcommand{\hc}{{HC}}
\newcommand{\tens}[1]{%
  \mathbin{\mathop{\otimes}\displaylimits_{#1}}%
}

\newcommand{\lam}{\lambda}

\theoremstyle{plain}
\newtheorem{theorem}{Theorem}[section]
\newtheorem{lemma}[theorem]{Lemma}
\newtheorem{definition}[theorem]{Definition}
\newtheorem{corollary}[theorem]{Corollary}
\newtheorem{remark}[theorem]{Remark}
\theoremstyle{remark}
\newtheorem*{note}{Note}
\newtheorem*{ex}{Example}

\numberwithin{equation}{theorem}

\title[Harish-Chandra bimodules in $\mathrm{Rep}(GL_t)$]{Harish-Chandra bimodules in the Deligne category \boldmath{${\mathrm{Rep}(GL}$}$_t$\boldmath{$)$}}

\author{Alexandra Utiralova}

\begin{document}

\maketitle

\begin{abstract}
    In this paper we study the category of Harish-Chandra bimodules $HC_{\chi,\psi}$ in the Deligne category $\Rep(GL_t)$. In particular, we answer Question 3.25 posed in Pavel Etingof's paper \cite{E} and determine for which central characters $\chi$ and $\psi$ this category is not zero. 
\end{abstract}

\section{Introduction}

Representation theory in complex rank first started as an example in the paper by Deligne and Milne \cite{DM}, where the category $\Rep(GL_t)$ for $t$ not necessarily integer was introduced. It was further developed in later papers by Deligne, where he introduced  the categories $\Rep(O_t), \Rep(Sp_{2t})$ and $\Rep(S_t)$, interpolating the categories of representations of the groups $O_n, Sp_{2n}$ and $S_n$ correspondingly, and also suggested the ultraproduct realization of these categories \cites{D1, D2}. 

The goal of this paper is to study the categories $HC_{\chi,\psi}$ of
Harish-Chandra bimodules for the Lie algebra $\gl_t$ in the Deligne
category $\Rep(GL_t)$ where $t$ is a generic complex number, which
interpolate the categories of Harish-Chandra bimodules for $\gl_n(\C)$ with
fixed central characters to non-integer values of $n$. Namely, we
determine for which values of central characters this category is not
zero. This answers Question 3.25 in \cite{E}.

Harish-Chandra bimodules for $GL_t$ were defined in \cite{E} as follows. {A $(\gl_t,\gl_t)$-bimodule $M\in\Ind(\Rep(GL_t))$ is a Harish-Chandra bimodule if it is finitely generated (i.e. it is a quotient of $U(\gl_t)\otimes U(\gl_t)\otimes X$ for some $X\in\Rep(GL_t)$), the action of the diagonal copy of $\gl_t\subset \gl_t\oplus \gl_t$ is natural and  both copies of $Z(U(\gl_t))$ act locally finitely on $M$.} The category $\hc_{\chi,\psi}$ is then the category of Harish-Chandra bimodules on which the left copy of the center $Z(U(\gl_t))$ acts by a central character $\chi$ and the right one acts by $\psi$.

Our main result is Theorem \ref{t13}, which shows that $\hc_{\chi, \psi}$ is nonzero if and only if $\chi(u)-\psi(u)= \sum_{i=1}^r e^{b_i u} - \sum_{i=1}^s e^{c_i u} $  (central characters for $\gl_t$ and the generating function notation for them are defined in \ref{d1}). 

\subsection*{Acknowledgements.}
I would like to thank Pavel Etingof for suggesting this problem to me and for all the valuable discussions we had about it.

\section{Preliminaries}

\subsection{Notations}

From now on let $\mathbb{k}\coloneqq \overline\Q$, let $\mathfrak h$, $\mathfrak b$ and $\n_{{-}}$ denote the standard Cartan subalgebra, Borel subalgebra of upper-triangular matrices and the subalgebra of strictly lower-triangular matrices in $\gl_n(\mathbb{k})$ correspondingly,  for any $\lam\in\h^*$ let $\ck_\lam$ be the extension to $\mathfrak b$ of the one-dimensional $\h$-module corresponding to $\lam$. Let $W\simeq S_n$ be the Weyl group of $\gl_n(\mathbb{k})$,  let $\Lambda^+$ be the lattice of integral dominant weights and $\rho$ be the half sum of all positive roots of $\gl_n(\mathbb{k})$. We denote by $V^\nn$ the $n$-dimensional defining representation of $GL_n$. By $\Delta$ we will always mean the coproduct map.

\subsection{Basic results and definitions}

\begin{definition}
The Deligne category $\Rep(GL_t)$ is the Karoubian 
envelope (formally adjoining images
of idempotents) of the rigid symmetric $\C$-linear tensor category generated by a single object $V$ of dimension $t\in \C$, such that $\En(V^{\otimes m})=\C[S_m]$ for all $m\ge 1$. 
\end{definition}

Let us denote the object $V^{\otimes r}\otimes (V^*)^{\otimes s}\in\Rep(GL_t)$ by $[r,s]$.

The following theorem is copied from \cite{E}, Theorem 2.9.

\begin{theorem}\cite{D2}  \label{tuniversalprop}
The category $\Rep(GL_t)$ has the following universal property:
if $\mathcal D$ is a symmetric tensor category then isomorphism classes of (possibly
non-faithful) symmetric tensor functors $\Rep(GL_t) \to \mathcal D$ are in bijection with isomorphism classes of objects of dimension $t$, via
$F \mapsto F([1, 0])$.
\end{theorem}

It turns out that for non-integer values of $t$ the categories $\Rep(GL_t)$ are abelian and semisimple (see for example \cite{EGNO}, Subsection 9.12).

We will now state the result showing that $\cc \coloneqq \Rep(GL_t)$ can be constructed as a subcategory in the ultraproduct of the categories $\cc_n$, where $\cc_n$ is the category of finite-dimensional representations of $GL_n$. For more details on the ultrafilters and ulraproducts  see \cite{HK} or \cite{S}. A very detailed and nice explanation of the following construction can be found in \cite{K}. The original statement for transcendental $t$ is due to Pierre Deligne \cite{D2} (but is left without proof). And the similar statement for all values of $t$ (requiring passing to positive characteristics) was proved in \cite{H} by Nate Harman.

Let $\cf$ be a nonprincipal ultrafilter on $\mathbb N$. We will fix some isomorphism of fields $ \prod_\cf \mathbbm k\simeq\mathbb C$.

\begin{theorem}{ \cite{D2}, \cite{H}}
\label{t1}
$\cc$ is equivalent to the full subcategory $\widetilde\cc$ in $\prod_{\mathcal F} \cc_n$ generated by $\widetilde V\coloneqq \prod_\cf V^{(n)}$ under the operations of taking duals, tensor products, direct  sums, and direct summands if $t$ is the image of $\prod_\cf  n$  under the isomorphism $\prod_\cf  \mathbbm{k} \simeq \C$.
\end{theorem}

The proof of this statement is almost identical to the proof of Theorem 1.1 in \cite{H} or Theorem 1.4.1 in \cite{HK}. 

\begin{proof}
Clearly, the categorical dimension of $\widetilde V$ is $t$. Therefore, by Theorem \ref{tuniversalprop}, we get a symmetric tensor functor $F:\cc \to \widetilde \cc$ with $F([1,0]) = \widetilde V$. Now, the category $\Rep(GL_t)$ is generated by $[1,0]$ under the operations of taking duals, tensor products, direct  sums, and direct summands. Thus, $F$ is essentially surjective. It is left to show that it is fully faithful. Since both $\cc$ and $\widetilde \cc$ are Karoubian envelopes of additive categories generated by $[r,s]$ and $\widetilde V^{\otimes r}\otimes(\widetilde V^*)^{\otimes s} $ correspondingly, it is enough to check that $F$ induces an isomorphism of algebras 
$$
\En_\cc([r,s]) \to \En_{\widetilde \cc}(\widetilde V^{\otimes r}\otimes(\widetilde V^*)^{\otimes s}).
$$
But this is an easy consequence of the Schur-Weyl duality, since both algebras are isomorphic to the walled Brauer algebra $B_{r,s}(t)$ (it is ensured for the ultraproduct, since it holds for $\En_{GL_n}((V^\nn)^{\otimes r}\otimes ((V^\nn)^*)^{\otimes s})$  for all  $n>r+s$).
\end{proof}

\begin{remark}
Clearly, one can only obtain transcendental numbers $t$ as the image of $\prod_{\cf} n$. To get this construction for algebraic $t$, one would need to consider the representations of $GL_n$ over some fields of positive characteristic (see \cite{H}).

However, applying automorphisms of $\C$ over $\mathbbm k$, one can show that any transcendental $t$ can be obtained in this manner.
\end{remark}

So, from now on we assume for simplicity that $t$ is non-algebraic. We expect, however, that similar results hold for all non-integer values of $t$. 

Let us denote by $\Ind(\cc)$ the category of ind-objects (i.e. filtered colimits of regular objects) of $\cc$.

\begin{definition}\label{dliealg}
Let $\g=\gl_t=V\otimes V^*$. It is a Lie algebra in $\cc$. Let us denote the commutator map $\g\otimes \g \to \g$ by $c$. We define its universal enveloping algebra $U(\g)\in \Ind(\cc)$ as a quotient of the tensor algebra $T(\g)\coloneqq \bigoplus_{k=0}^\infty \g^{\otimes k}$ by the ideal generated by the image of the map
$$
r:\g\otimes \g \to \g\oplus (\g\otimes \g) \subset T(\g),
$$
$$
r = c\oplus (\sigma_\g - id_\g),
$$
where $\sigma_\g: \g\otimes\g\to\g\otimes\g$ is the permutation of tensor factors. 
\end{definition}

For any object $X=\prod_\cf X^\nn$ of $\cc$ we can define an action map
$$
a_X: \g\otimes X\to X
$$
as the ultraproduct $\prod_\cf a_{X^\nn}$, where $a_{X^\nn}:\gl_n(\ck)\otimes X^\nn\to X^\nn$ is the natural action of $\gl_n(\ck)$ on $X^\nn$. Clearly, $a_X$ is a Lie algebra action, i.e. 
$$
(a_X\oplus (a_X\circ (id_\g\otimes a_X)))\circ (r\otimes id_X) = 0
$$
as a map from $\g\otimes\g\otimes X$ to $X$ (where $r$ is the map $\g\otimes \g \to \g\oplus(\g\otimes \g)$ defined in \ref{dliealg}). That is, $a_X$ induces the action of $U(\g)$ on $X$. We will refer to this action as the \textit{natural} action of $\g$ (or $U(\g)$).

 Clearly, the natural action of $\g$ on $\cc$ extends to $\Ind(\cc)$.

 We define the Poincar\'e–Birkhoff–Witt (or PBW) filtration $F$ on $U(\g)$ as the image of $T^i(\g)\coloneqq \g^{\otimes i}$ under the quotient map $T(\g)\to U(\g)$. We have 
$$
F_iU(\g)\coloneqq \prod_\cf F_i U(\gl_n(\mathbb k)),
$$
where we abuse the notation and denote by $F$ the PBW-filtration on $U(\gl_n(\mathbb k))$ as well.

The center of $U(\g)$
$$
Z(U(\g))= U(\g)^{GL_t}=\Hom(\mathbb{1},U(\g))=\bigcup_i \Hom(\mathbb{1},F_i U(\g))
$$
is a filtered algebra in the category of $\C$-vector spaces. We have 
$$
F_i Z(U(\g))\coloneqq  Z(U(\g))\cap F_i U(\g) = \prod_\cf (Z(U(\gl_n(\mathbb{k})))\cap F_iU(\gl_n(\mathbb{k}))).
$$

The Harish-Chandra isomorphism tells us that $Z(U(\gl_n(\ck))$ is isomorphic to the algebra of symmetric polynomials $\ck[\mathfrak h^*]^W$ with filtration given by the degree. Given a central element $C$ we recover the corresponding symmetric polynomial $p$ by looking at the action of $C$ on the Verma module $M_\lam \coloneqq U(\gl_n(\ck))\tens{U(\mathfrak b)}\ck_{\lam-\rho}$ with the highest weight $\lam-\rho\in \h^*$. We have $p(\lam) = C|_{M_\lam}$.

\begin{remark}
The shift of the highest weight in the definition of $M_\lam$ is needed to replace the dot-action of $W$ on $\h^*$ with the usual action.
\end{remark}

Symmetric polynomials in $n$ variables are freely generated (as an algebra) by the first $n$ power sums polynomials $p_k \coloneqq \sum_i x_i^k$. For any $k$ and any $n$ we define the element $C_k\in Z(U(\gl_n(\mathbb{k})))$ to be the image of $p_k$ under the isomorphism $\mathbb k[x_1,\ldots,x_n]^{S_n}\to Z(U(\gl_n(\mathbb{k})))$. I.e. $C_k$ is the central element, which acts on each $M_\lam$ via the constant $\sum_i \lam_i^k$.

We get that $Z(U(\gl_n(\ck)))\simeq \ck[C_1,\ldots, C_n]$ with $\deg C_k=k$. And therefore, 
$$
Z(U(\g))\simeq \mathbb C[C_1, C_2, \ldots],
$$
$$
C_i \coloneqq  \prod_\cf C_i  \text{ (by the abuse of notation).}
$$

\begin{definition}
\label{d1}
A central character is an algebra homomorphism 
$$
\psi: Z(U(\g))\to \mathbb C.
$$ 
By the result above, $\psi$ is completely determined by the numbers $\psi_k=\psi(C_k)$. For convenience let us adopt the (exponential) generating function notation $\psi(u)=\frac{1}{(e^u-1)} \sum \frac{1}{k!}\psi_k u^k\in\mathbb C((u))$, where we put $\psi_0=1$.
\end{definition}

\begin{remark}
We divide the generating function by $(e^u-1)$ only for the reason that it yields better looking formulas, and we treat the factor $\frac{1}{(e^u-1)}$ formally.
\end{remark}

For each $\gl_n(\mathbb{k})$ let us fix a central character $\psi^{(n)}: Z(U(\gl_n(\mathbb{k})))\to\mathbb{k}$. It is determined by $n$ numbers $\psi_k^{(n)}=\psi^{(n)}(C_k)$ with $1\le k \le n$. However, since the elements $C_k$ are defined for all $k$, we can define the (exponential) generating function $\psi^{(n)}(u)= \frac{1}{(e^u-1)}\sum \frac{1}{k!} \psi^{(n)}_k u^k\in \mathbb{k}((u))$. Due to the algebraic independence of $\{C_k\}_{k\le n}$ in $Z(U(\gl_n(\mathbb{k})))$ we can choose $\psi_k^{(n)}$ to be an arbitrary number for $n>k$. Thus, any central character $\psi$ of $\g$ can be obtained as an ultraproduct of central characters $\psi^\nn$ of $\gl_n(\ck)$ (i.e. $\psi_k = \prod_\cf \psi^\nn_k$). We write $\psi = \prod_\cf \psi^\nn$.

\section{The category $\hc_{\chi,\psi}$}

\subsection{Definitions and notations}

We will henceforth consider objects with the action of $\g\oplus\g$. Let us denote by $\g_l,~ \g_r$ and $\g_d$ the left, the right and the diagonal copy of $\g$ inside $\g\oplus\g$ (with the diagonal copy being the image of the map $id_\g\oplus id_\g: \g\to \g\oplus \g$). Moreover, we will freely switch between the left action of $\g\oplus\g$ and the double action of $\g$ both on the left and on the right, the first given by the action of $\g_l$ and the second by minus the action of $\g_r$.

Clearly, any left $\g$-module $M$ in $\Ind(\cc)$ with the action map $f: \g\otimes M\to M$ has a unique action of $\g\oplus\g$, s.t.  $\g_l$ acts by the original action and the diagonal copy $\g_{d}$ acts naturally, i.e. via the map $a_M$. Indeed, one defines the new action map $(\g\oplus\g)\otimes M\to M$ as $f\oplus (a_M-f)$.   

\begin{definition}\label{dhcbimodules}
A Harish-Chandra bimodule for $GL_t$ is a $\g\oplus \g$-module  $M\in\Ind(\cc)$, such that
\begin{enumerate}
    \item $M$ is finitely generated, i.e. it is a quotient of $(U(\g)\otimes U(\g))\otimes X$ for some $X\in\cc$;
    
    \item $\g_d$ acts naturally (i.e. via $a_M$) on $M$;
    
    \item the center $Z(U(\g)\otimes U(\g))$ acts locally finitely on $M$, i.e. $\mathrm{Ann}_{ Z(U(\g)\otimes U(\g))}M$ is an ideal of finite codimension.
\end{enumerate}

\end{definition}

Any maximal ideal in $Z(U(\g)\otimes U(\g))$ (i.e. ideal of codimension 1) corresponds to a pair of central characters $\chi,\psi: Z(U(\g))\to \C$. Denote by $\hc_{\chi,\psi}$ the category of Harish-Chandra bimodules on which $Z(U(\g_l))$ acts via $\chi$ and $Z(U(\g_r))$ acts via $\psi$.

\begin{ex}
Let $U_\chi \coloneqq U(\g)/(z-\chi(z))$, where $z$ runs over all elements in the center. Clearly, $U_\chi\in\hc_{\chi,\chi}$.
\end{ex}
\begin{note}
For each $i$  we have $F_i U_\chi=\prod_\cf F_i U(\gl_n(\mathbb{k}))/(z-\chi^{(n)}(z))$, where $\chi = \prod_\cf \chi^{(n)}$.\footnote{We will abuse the notation and denote by $U_\chi$ the quotient $U(\gl_n(\ck))/(z-\chi(z))$ too, when it is clear from the context which one we refer to.}
\end{note}

Clearly, every irreducible Harish-Chandra bimodule must lie in one of the categories $\hc_{\chi,\psi}$, so they are interesting to study.

\begin{remark}\label{raboutclassichc}
Note that for regular Harish-Chandra bimodules (for $GL_n$) condition $(2)$ of Definition \ref{dhcbimodules}  translates to the fact that one can integrate the action of the diagonal copy $(\gl_n)_d$ on $M$ to the action of $GL_n$; and condition $(3)$ is equivalent to $M$ having a \textit{finite $K$-type}, which means that for any simple $GL_n$ representation $L$ the multiplicity space $\Hom_{GL_n}(L,M)$ is finite-dimensional (see \cite{BG}, Proposition 5.3). However, in $\Rep(GL_t)$ the latter is no longer true. Harish-Chandra bimodules of finite $K$-type are studied to some extent in \cite{U}. One can also find there some non-trivial examples of irreducible Harish-Chandra bimodules.
\end{remark}

For any object $X\in\cc$ the tensor product $U_\psi\otimes X$ is naturally a left $\g$-module with $\g$ acting on $U_\psi$ by multiplication on the left and on $X$ naturally (so it is easy to deduce that the right action of $\g$ is by (minus) multiplication on the right and affects only the $U_\psi$ part of the tensor product). Let $N(\chi,\psi, X):=(U_\psi\otimes X)_\chi$ denote the  quotient  $(U_\psi\otimes X)/(z-\chi(z))(U_\psi\otimes X)$, where we factor out by the action of the left copy of $Z(U(\g))$. Clearly, $N(\chi,\psi, X)\in\hc_{\chi,\psi}$.

\subsection{Bimodules $N(\chi,\psi,X)$}

\begin{lemma}(see \cite{E} Section 3.6) 
The category $\hc_{\chi,\psi}$ is nonzero iff  $N(\chi,\psi, X)\neq 0$ for some $X\in\cc$.
\end{lemma}
\begin{proof}
For any $X\in\cc$ we have that as $(\g,\g)$-bimodules $$N(\chi,\psi,X) \coloneqq U_{\chi}\tens{U(\g_{l})}(U_\psi\otimes X)\simeq(U_{\chi}\otimes U_{\psi}^{op})\tens{U(\g_{d})} X,$$ where $\g_d$ acts on the right on $U_\chi\otimes U_\psi^{op}$. 

Now suppose $M\in\hc_{\chi,\psi}$, then $M$ is finitely generated, i.e. it is generated by some subobject $X$ (that lies in $\cc$). There is a natural morphism $N(\chi, \psi, X)\to M$ whose image is the subbimodule of $M$ generated by $X$ (since $X$ generates $M$, it is surjective). Thus, if $M$ is nonzero then so is $N(\chi,\psi, X)$.
\end{proof}

\begin{corollary}
 The category $\hc_{\chi,\psi}$ is nonzero iff  $N(\chi,\psi, [r,s])\neq 0$ for some $[r,s]$.
\end{corollary}

So, we want to understand for which $\chi,\psi$ the bimodule $(U_\psi\otimes[r,s])_\chi$ is not zero. For this purpose let us first look at $(U_{\psi}\otimes V)_\chi$ and understand for which $\chi$ and $\psi$ it is not zero.

\subsection{The basic case for $\gl_n(\ck)$}
Let us look at the finite-dimensional case, namely $\gl_n(\mathbb{k})$. We have a homomorphism $U_\psi\to \En_\ck(M_\lam)$ for some $\lam=(\lam_1,\ldots, \lam_n)$ (determined up to permutation of $\lam_i$) which is injective by Duflo's theorem (which states that the annihilator of a Verma module is the ideal generated by the kernel of the corresponding central character, see \cite{Dix}, Theorem 8.4.3). So there is an injective morphism of $\gl_n(\mathbb{k})$-bimodules $$U_\psi\otimes V^\nn \hookrightarrow \Hom_{\mathbb k}(M_\lam, M_\lam\otimes V^\nn),$$ where the right action of $\gl_n(\mathbb{k})$ is on the source and the left action is on the target. Thus, it is enough to consider the action of the center on $M_\lam\otimes V^\nn$ (instead of considering the left action on $U_\psi\otimes V^\nn$). For generic (without nontrivial stabilizers in $W$) weight $\lambda$ we have $M_\lam\otimes V^\nn=\bigoplus_{i=1}^n M_{\lam+e_i}$, where $\lam+e_i=(\lam_1,\ldots,\lam_{i-1},\lam_i+1,\lam_{i+1},\ldots,\lam_n)$. When restricted to the direct summand $M_{\lam+e_l}$, each $C_k$ acts by $\sum_{i=1}^n \lam_i^k + (\lam_l+1)^k-\lam_l^k$.

Let $\Omega:=\frac{1}{2}(\Delta(C_2)-C_2\otimes 1 - 1\otimes 1)$. Then $$(2\Omega+1)|_{M_{\lam+e_l}} = 2\lam_l+1$$ and thus, $$(\Delta(C_k)-C_k\otimes 1)|_{U_\psi\otimes V^\nn}=(\Omega|_{U_\psi\otimes V^\nn}~+~1)^k-\Omega^k|_{U_\psi\otimes V^\nn}.$$ Let us denote $P_k(b)\coloneqq (b+1)^k-b^k$. Then, we have proved the following lemma:
\begin{lemma}
For generic central character $\psi$, i.e. such that $U_\psi$ acts on $M_\lam$ with generic $\lam$, we have: $$\Delta(C_k)|_{U_{\psi}\otimes V^\nn}-C_k\otimes 1|_{U_{\psi}\otimes V^\nn}=P_k(\Omega|_{U_\psi\otimes V^\nn}).$$
\end{lemma}

Now let us identify the $\mathbb k$-algebra $\En_{\mathbb k}(M_\lam)$ with $A\coloneqq \En_{\mathbb k}(U(\n_{{-}}))$ via the natural isomorphism of vector spaces $M_\lam$ and $U(\n_{{-}})$. The representations $U(\gl_n(\mathbb{k})) \to \En_{\mathbb{k} }(M_\lam)$ thus give us a family of algebra homomorphisms to $A$ depending on $\lam$ polynomially.  That is to say, we have a map of algebras $\varphi: U(\gl_n(\mathbb{k})) \to A \otimes  \mathbb{k}[x_1,\ldots,x_n]$, which composed with the quotient map $ \mathbb{k}[x_1,\ldots,x_n]\to  \mathbb{k}$ sending $x_i$ to $\lam_i$, and the natural isomorphism between $A$ and $\En_{\mathbb k}(M_\lam)$, gives us precisely the representation of $\gl_n(\mathbb{k})$ on $M_\lam$. Let us denote by $\mathfrak{m}_\lam$ the maximal ideal generated by $x_i-\lam_i, 1\le i\le n$.

The family of representations of $\gl_n(\mathbb{k})\oplus\gl_n(\mathbb{k})$ on $M_\lam\otimes V^\nn$ for varying $\lam$ (with the left copy acting on $M_\lam$ and the right copy acting on $V^\nn$) produces a map $\varphi\otimes\rho_{V^\nn}: U(\gl_n(\mathbb{k}))\otimes U(\gl_n(\mathbb{k})) \to A\otimes \mathbb{k}[x_1,\ldots,x_n]\otimes \En_{\mathbb k}(V^\nn)$. We have proved above that the image of $\Delta(C_k)-C_k\otimes 1 - P_k(\Omega)$ under $\varphi\otimes\rho_{V^\nn}$ has to lie in  $A\otimes(\bigcap_{\lam\text{ generic}} \mathfrak{m}_\lam)\otimes \En_{\mathbb k}(V^\nn)$. Since the set of $\lam\in \mathbb k^n$ such that some $\lam_i = \lam_j$ for $i\neq j$ (i.e. $\lam$ has a nontrivial stabilizer in $W$ and hence is non-generic by our definition) is closed in $\mathbb k^n$, the intersection $\bigcap_{\lam\text{ generic}} \mathfrak{m}_\lam$ is zero. Thus, $\Delta(C_k)-C_k\otimes 1 - P_k(\Omega)$ acts as zero on ${M_\lam\otimes V^\nn}$ for any (not necessary generic) $\lam$. 

\begin{corollary}\label{cactiononutensv}
For any central character $\psi$ of $U(\gl_n(\mathbb k))$
$$\Delta(C_k)|_{U_{\psi}\otimes V^\nn}-C_k\otimes 1|_{U_{\psi}\otimes V^\nn}=P_k(\Omega|_{U_\psi\otimes V^\nn}).$$
\end{corollary}

\begin{remark}
We have $\frac{1}{(e^u-1)} \sum \frac{1}{k!} P_k(b)u^k= e^{bu}$.
\end{remark}

Recall that $\Omega$ was defined as $\frac{1}{2}(\Delta(C_2)-C_2\otimes 1 - 1\otimes 1)$. The element $C_k\in Z(U((\gl_n)_r))$ acts on $U_\psi\otimes V^\nn$ as $C_k\otimes 1$, and $C_k\in Z(U((\gl_n)_l))$ acts on $U_\psi\otimes V^\nn$ as $\Delta(C_k)$. Thus, $\Omega$ acts on $(U_\psi\otimes V^\nn)_\chi$ as a constant $b=\frac{1}{2}(\chi_2-\psi_2-1)$. So, by Corollary \ref{cactiononutensv}, $(U_\psi\otimes V^\nn)_{\chi}\neq 0$ only when $\chi(u)-\psi(u)=e^{bu}$.

Thus, for each $\gl_n(\mathbb{k})$ we must have $\chi^{(n)}(u)-\psi^{(n)}(u)=e^{b_n u}$ for some $b_n\in \mathbb k$. So, in the case of $\gl_t$, $\chi(u)-\psi(u)=e^{b u}$ for $b=\prod_\cf b_n$.

\subsection{The basic case for $\gl_t$}

We have just proved the following statement.

\begin{lemma}
For any central characters $\psi, \chi$ of $U(\g)$ the bimodule $(U_\psi\otimes V)_\chi$ is nonzero only if
$$
\chi(u)-\psi(u) = e^{bu}
$$
for some $b\in\C$.
\end{lemma}

\begin{lemma}
\label{l1}
For any finitely generated $(\g,\g)$-bimodule $M$ in $\Ind(\cc)$ $$(\Delta(C_k)-C_k\otimes 1)|_{M\otimes V}=P_k(\Omega|_{M\otimes V}).$$
\end{lemma}
\begin{proof}

We have proved this statement for $M=U_\psi$. It is easy to see that the same proof works for the module $M=U_\psi\otimes U(\g)\otimes X$, where $X\in\cc$ and $U(\g_l)$ acts only on the $U_\psi$ part of the tensor product. $X$ is represented by some sequence of $GL_n$-modules $(X_1, X_2,\ldots)$, so we can pass to a finite dimensional case. There $U_\psi\otimes U(\gl_n(\mathbb k))\otimes X_n\otimes V\hookrightarrow \bigoplus_{l=1}^n \Hom_{\mathbb k}(M_\lam, M_{\lam+e_l})\otimes U(\gl_n(\mathbb k))\otimes X_n$ and $\gl_n(\mathbb{k})_l$ acts only on the $\Hom$-part of the tensor product, moreover, only on the target, i.e. on $M_{\lam+e_{l}}$, and we can repeat the proof from above. Thus, we have obtained the following: 

$$(\Delta(C_k)-C_k\otimes 1)|_{(U_{\psi}\otimes U(\g)^{op}\otimes X) \otimes V}=P_k(\Omega|_{(U_{\psi}\otimes U(\g)^{op}\otimes X)\otimes V}).$$

For any $n\in\Z_{>0}$ the universal enveloping algebra $U(\gl_n(\mathbb{k}))$ embeds into $\bigoplus_{\chi^{(n)}} U_{\chi^{(n)}}$, where the sum runs over all central characters $\chi^{(n)}: Z(U(\gl_n(\mathbb{k})))\to \mathbb k$. Therefore, the same holds for $U(\g)$. Hence, the statement of the lemma holds for $M=U(\g)\otimes U(\g)^{op} \otimes X$, where $X\in\cc$. 

Finally, by definition, any finitely generated $U(\g)$-module is a quotient of \\ $U(\g)\otimes U(\g)\otimes X$ with $X\in\cc$ and $U(\g_l)$ acting only on the leftmost part of this tensor product. This ends the proof.
\end{proof}

\begin{remark}
\label{rem}
A similar reasoning shows that for any finitely generated $U(\g)$-module $M$ in $\Ind(\cc)$
$$
(\Delta(C_k)-C_k\otimes 1)|_{M\otimes V^{*}}=\overline P_k(\Omega|_{M\otimes V^*}),
$$
where $\overline P_k(c)=(c-1)^k-c^k$.
\end{remark}

\begin{note}
It is easy to see that $\frac{1}{(e^u-1)}\sum \overline P_k(c)u^k =  -e^{(c-1)u}$.
\end{note}

\subsection{The general case}
Now we are ready to consider the action of the central elements on $U_\psi\otimes[r,s]$. Note that the central element $C$ of $Z(U(\g_{r}))$ acts on this module as $(C\otimes 1)_{U_{\psi}, [r,s]}$, and if it is considered as an element of $Z(U(\g_{l}))$ then it acts as $\Delta(C)_{U_{\psi}, [r,s]}$ meaning that the first tensor factor acts on $U_\psi$ and the second acts on $[r,s]$.

Let us adopt some convenient notation. Let $A$ be any element in $Z(U(\g)\otimes U(\g))$. Denote by $A_j$ its image under the following map $\tau_j: U(\g)\otimes U(\g)\to U(\g)^{\otimes (m+1)}$: 
$$
\tau_j=(\Delta^{j-1}\otimes \text{id})\otimes \underbrace{1\otimes\ldots\otimes 1}_{m-j}. 
$$
We also denote by $A_j$ the corresponding operator acting on $X_0\otimes X_1\otimes\ldots\otimes X_m$, where each $X_i$ is a $U(\g)$-module.

Let us view $U_\psi\otimes[r,s]$ as the string of tensor factors 
$$
U_\psi\otimes V\otimes\ldots\otimes V\otimes V^*\otimes\ldots\otimes V^*.
$$
Then 

$$
    (\Delta(C_k)-C_k\otimes 1)_{U_{\psi}, [r,s]}=\Delta^{r+s}(C_k)-(C_k\otimes 1)_1 =
$$

$$
    =(\Delta^{r+s-1}\otimes \text{id})(\Delta(C_k)-C_k\otimes 1)+\Delta^{r+s-1}(C_k)\otimes 1 - (C_k\otimes 1)_1 =
$$

$$
    =(\Delta(C_k)-C_k\otimes 1)_{r+s}+ \Delta^{r+s-1}(C_k)\otimes 1 - (C_k\otimes 1)_1=
$$

$$
    =\ldots = 
$$

$$
    = (\Delta(C_k)-C_k\otimes 1)_{r+s}+(\Delta(C_k)-C_k\otimes 1)_{r+s-1}+\ldots + (\Delta(C_k)-C_k\otimes 1)_{1}.
$$

By Lemma \ref{l1} and Remark \ref{rem}, this is equal to
$$
\overline P_k(\Omega_{r+s})+\ldots+\overline P_k(\Omega_{r+1})+ P_k(\Omega_r)+\ldots P_k(\Omega_1).
$$

\begin{theorem}
$(U_\psi\otimes[r,s])_\chi\neq 0$ only if there exist numbers $b_1,\ldots,b_r,\\ c_1,\ldots,c_s\in\C$ such that
$$
\chi(u)-\psi(u) =
\sum_{i=1}^r e^{b_i u} - \sum_{i=1}^s e^{(c_i-1)u}.
$$
\end{theorem}
\begin{proof}

Let us consider $\En_{\mathbb k}((U_\psi\otimes[r,s])_\chi)$. It is a nonzero algebra and there is a homomorphism from $\C[\Omega_1,...,\Omega_{r+s}]$ to it. Its image is a nonzero finitely generated commutative algebra where
$$
\overline P_k(\Omega_{r+s})+\ldots+\overline P_k(\Omega_{r+1})+ P_k(\Omega_r)+\ldots P_k(\Omega_1)= \chi_k-\psi_k.
$$
Thus, by the Nullstellensatz, there exists a maximal ideal in this algebra, i.e. numbers $b_1,\ldots,b_r,  c_1,\ldots, c_s$, such that 
$$
\overline P_k(c_s)+\ldots+\overline P_k(c_{1})+ P_k(b_r)+\ldots P_k(b_1)= \chi_k-\psi_k.
$$
This ends the proof.
\end{proof}

\begin{corollary}
\label{c7}
 If $\hc_{\chi,\psi}\neq 0$ then there exist numbers $b_1,\ldots,b_r,\\ c_1,\ldots,c_s\in\C$ such that
$$
\chi(u)-\psi(u) = \sum_{i=1}^r e^{b_i u} - \sum_{i=1}^s e^{(c_i-1)u} .
$$
\end{corollary}

\subsection{The reverse direction: constructing a nonzero object in $\hc_{\chi,\psi}$}
Now we want to prove the converse, i.e. if such numbers exist, then the category is nonzero. The remainder of this paper will be devoted to proving the following theorem.

\begin{theorem}
\label{t8}
Suppose there exist numbers $b_1,\ldots,b_r,c_1,\ldots,c_s\in\C$ such that
$$
\chi(u)-\psi(u) =
\sum_{i=1}^r e^{b_i u} - \sum_{i=1}^s e^{(c_i-1)u}.
$$

Then $\hc_{\chi,\psi}\neq 0$. 
\end{theorem}

To prove this we are going to show that $(U_{\psi}\otimes X)_{\chi}$ is nonzero for some $X\in \cc$.

\begin{lemma}
\label{l8}
Let $b_i=\prod_{\cf}~b_i^{(n)}$, for $1\le i\le r$ and $c_i=\prod_{\cf}~c_i^{(n)}$, for $1\le i \le s$, with $b_i^{(n)},c_i^{(n)}\in \mathbb k$.

Then for any $\psi(u)$ there exists a presentation of it as an ultraproduct of $\psi^{(n)}(u)$ - central characters of $\gl_n(\mathbb{k})$ - so that $\psi^{(n)}(u) = 0$ for $n\le r+s$ and when $n>r+s$ then $\gl_n(\mathbb{k})$ acts with central character $\psi^{(n)}$ on some $M_{\mu^{(n)}}$, where $$\mu^{(n)}_i= b^{(n)}_i, ~~~1\le i\le r
$$
$$
\mu^{(n)}_{n-j+1}=c^{(n)}_j, ~~~1\le j\le s
$$ 

\end{lemma}
\begin{proof}

Let us take an arbitrary presentation of the numbers $\psi_k$ as an ultraproduct: $\psi_k = \prod_\cf~\phi_k^{(n)}$.

We will prove that for a fixed $k$ we can change finitely many of the numbers $\phi_k^{(n)}$, namely those with $n < k+s+r$ so that the resulting central characters satisfy the condition above. 

For a fixed $n$ consider the following equations on $m\coloneqq n-r-s$ variables $x_i$ for $k=1,\ldots , m$:
\begin{equation}
  \label{f1}
(b_1^{(n)})^k+\ldots+(b_r^{(n)})^k+(c_1^{(n)})^k+\ldots+(c_s^{(n)})^k+ x_1^k+\ldots+x_{m}^k = \phi^{(n)}_k.  
\end{equation}

They are equations on the first $m$ power sums of $x_i$.  The ring of symmetric polynomials in $m$ variables is freely generated by the first $m$ power sums, so these equations determine a point in the maximal spectrum of $\ck[x_1,\ldots, x_m]^{S_n}$. The inclusion $\ck[x_1,\ldots,x_m]^{S_n}\hookrightarrow \ck[x_1,\ldots,x_m]$ induces a surjective map on the maximal spectra (the quotient map by the $S_n$-action), thus, there exists a solution $x_i = a_i^{(n)}$ for $1\le i\le m$.

Put $$\mu_1^{(n)}=b_1^{(n)}, \ldots,  \mu_r^{(n)}=b_r^{(n)},
$$
$$
\mu_{r+1}^\nn=a_1^{(n)},\ldots,  \mu_{n-s}^{(n)} = a_{n-r-s}^{(n)},
$$
$$
\mu_{n-s+1}^{(n)}=c_s^{(n)}, \ldots,  \mu_n^{(n)}=c_1^{(n)}.
$$

Consider the central character $\psi^{(n)}$ corresponding to the $\mu^{(n)}$ above. Then for $1\le k\le n-r-s$ we have  
$$
\psi^{(n)}_k  \coloneqq \sum_{i=1}^n (\mu_i^{(n)})^k = \sum_{i=1}^r (b_i^\nn)^k + \sum_{i=1}^{m} (a_i^\nn)^k + \sum_{i=1}^s (c_i^\nn)^k= \phi^{(n)}_k\text{ (by } \ref{f1}). $$

Putting $\psi^{(n)}=0$ for $n\le r+s$, we see that for a fixed $k$ we have $\psi^{(n)}_k = \phi^{(n)}_k$ for all $n\ge k+r+s$  and hence, $\prod_\cf \psi_k^{(n)} = \prod_\cf \phi_k^{(n)}= \psi_k.$
\end{proof}

\begin{lemma}
\label{l9}
Let $n>r+s$ and $$\lam = \mu+e_1+\ldots+e_r-e_{n-s+1}-\ldots-e_n.$$ Suppose $Z(U(\gl_n(\mathbb{k})))$ acts on $M_\lam$ with character $\chi$ and on $M_\mu$ with character $\psi$. Then $$\chi(u)-\psi(u) = e^{\mu_1 u}+\ldots+e^{\mu_r u}- e^{(\mu_{n-s+1}-1)u}-\ldots - e^{(\mu_n - 1)u}.$$
\end{lemma}

\begin{proof}
 The element $C_k$ acts on $M_{\mu+e_l}$ as $\sum_j \mu_j^k + P_k(\mu_l)$ and on $M_{\mu-e_l}$ as \\ $\sum_j \mu_j^k + \overline{P}_k(\mu_l)$. Thus, $$C_k|_{M_{\mu+e_l}}-C_k|_{M_\mu} = P_k(\mu_l)$$ and $$C_k|_{M_{\mu-e_l}}-C_k|_{M_\mu} = \overline P_k(\mu_l).$$

We put $\lam^{[0]} \coloneqq \lam, \lam^{[i]} \coloneqq \lam^{[i-1]}-e_i, 1\le i\le r$ and $ \mu^{[i]}\coloneqq\mu^{[i-1]}-e_{n-i+1}, 1\le i\le s$ with $\mu^{[0]}\coloneqq\mu$ (thus $\mu^{[s]} = \lam^{[r]}$):
$$
\mu^{[0]} = (\mu_1,\ldots,\mu_n),
$$
$$
\mu^{[1]} = (\mu_1,\ldots,\mu_{n-1},\mu_n-1),
$$
$$
\mu^{[2]} = (\mu_1,\ldots,\mu_{n-2},\mu_{n-1}-1,\mu_n-1),
$$
$$
\cdots
$$
$$
\mu^{[s]} = \lam^{[r]} = (\mu_1,\ldots,\mu_{n-s}, \mu_{n-s+1}-1,\ldots, \mu_{n}-1),
$$
$$
\lam^{[r-1]} = (\mu_1,\ldots,\mu_{r-1},\mu_r+1,\mu_{r+1},\ldots, \mu_{n-s},\mu_{n-s+1}-1,\ldots,\mu_n-1),
$$
$$
\cdots
$$
$$
\lam^{[1]} = (\mu_1, \mu_2+1,\ldots,\mu_r+1,\mu_{r+1},\ldots,\mu_{n-s},\mu_{n-s+1}-1,\ldots,\mu_n-1)
$$
$$
\lam^{[0]} = (\mu_1+1,\ldots,\mu_r+1,\mu_{r+1},\ldots,\mu_{n-s},\mu_{n-s+1}-1,\ldots,\mu_n-1).
$$

Now let $\chi^{[i]}$ be the central character corresponding to $\lam^{[i]}$ and $\psi^{[i]}$ be the central character corresponding to $\mu^{[i]}$. Thus, $\chi^{[i]}_k-\chi^{[i+1]}_k = P_k(\lam^{[i+1]}_{i+1}) = P_k(\mu_{i+1})$ and $\psi^{[i+1]}-\psi^{[i]}=\overline P_k(\mu^{[i]}_{n-i})=\overline P_k(\mu_{n-i})$. Summing up these equations for all $\chi$'s and $\psi$'s  we obtain $$\chi_k-\psi_k= P_k(\mu_1)+\ldots+P_k(\mu_r)+\overline P_k(\mu_{n-s+1})+\ldots+\overline P_k(\mu_n),$$ which leads to the desired relation for generating functions $\chi(u)$ and $\psi(u)$.
\end{proof}

\begin{corollary}{(of Lemmas \ref{l8} and \ref{l9})}
\label{c10}

 Suppose $$
 \chi(u)-\psi(u) = \sum_{i=1}^r e^{b_i u} - \sum_{i=1}^s e^{(c_i-1)u} 
 $$
 for some $b_i\in \C, 1\le i\le r, c_j\in\C, 1\le j\le s$. Then there exist presentations of $\chi$ and $\psi$ as ultraproducts of $\chi^{(n)}$ and $\psi^{(n)}$ correspondingly, so that for any $n>r+s$ the algebra $U(\gl_n(\mathbb{k}))$ acts on some $M_{\lam^{(n)}}$ with central character $\chi^{(n)}$ and if $\mu^{(n)} = \lam^{(n)}-e_1-\ldots-e_r+e_{n-s+1}+\ldots + e_n$, then $U(\gl_n(\mathbb{k}))$ acts on $M_{\mu^{(n)}}$ with central character $\psi^{(n)}$.
\end{corollary}
\begin{proof}
We apply Lemma \ref{l8} to $\psi(u),b_i,c_j$ to obtain central characters $\psi^{(n)}$ and weights $\mu^{(n)}$, satisfying the conditions of the lemma. Then for $n>r+s$, we put $\lam^{(n)}= \mu^{(n)}+e_1+\ldots+e_r-e_{n-s+1}-\ldots - e_n$ and denote by $\chi^{(n)}$ the central character corresponding to $\lam^{(n)}$.

We use Lemma \ref{l9} for $\lam^{(n)}$ and $\mu^{(n)}$ to see that $$\chi^{(n)}(u)-\psi^{(n)}(u) = e^{b_1^{(n)} u}+\ldots+e^{b_r^{(n)} u}- e^{(c_1^{(n)}-1)u}-\ldots- e^{(c_s^{(n)}-1)u}.$$

We put $\chi^{(n)} = 0$ for $n\le r+s$. If now $\widetilde{\chi}(u)=\prod_\cf \chi^{(n)}(u)$ then $$\widetilde{\chi}(u)-\psi(u) = e^{b_1 u}+\ldots+e^{b_r u}-e^{(c_1-1)u}-\ldots - e^{(c_s-1)u}$$ and thus $\widetilde{\chi}(u)=\chi(u)$ and we are done.
\end{proof}

\begin{lemma}
\label{l11}
Let $\lam,\mu,\text{ be such that  } \lam-\mu \in\Lambda^+$ and suppose $X$ is a finite-dimensional $\gl_n(\mathbb{k})$-module, with maximal weight  $\lam-\mu$, let $\chi$ be the central character corresponding to $\lam$ and $\psi$ be the central character corresponding to $\mu$. Then we have
$$
(U_\psi\otimes X)_\chi \neq 0.
$$
Moreover, if $F_0 U_\psi$ is the zeroth filtered component, that is the span of $1$, the quotient $(F_0 U_\psi\otimes X)_\chi$ is nonzero.
\end{lemma}
\begin{proof}
Consider a natural surjective map of $U(\gl_n(\mathbb{k}))$-modules
$$
U_\psi\otimes X\to M_\mu\otimes X\to 0,
$$
which sends $u\otimes x$ to $u v_\mu\otimes x$, where $v_\mu$ is the highest weight vector in $M_\mu$.

Taking tensor product with  $U_\chi$ over $U(\gl_n(\ck))$ for some $\chi$,  $$\chi:Z(U(\gl_n(\mathbb{k})))\to\ck,$$ is right exact, thus we have

$$
(U_\psi\otimes X)_\chi\to (M_\mu\otimes X)_\chi\to 0.
$$

So it is left to prove that $(M_\mu\otimes X)_\chi\neq 0$ and that the image of $v_\mu\otimes X$ is nonzero.  

We note that $\lam$ is maximal among weights of $M_\mu\otimes X$. Using the standard argument, i.e. taking the sum of all submodules of $M_\mu\otimes X$ (that are naturally objects of the category $\mathcal O)$, that don't contain weight $\lam$, one can show that $M_\mu\otimes X$ has a simple quotient-module $L$ with highest weight $\lam$. And if $x_{\lam-\mu}$ is a vector of maximal weight in $X$ then the highest weight vector of $L$ is the image of $v_{\mu}\otimes x_{\lam-\mu}$. 

 Since $L$ is invariant under taking tensor product with $U_\chi$ over $U(\gl_n(\ck))$, there is a surjective map
$$
(M_\mu\otimes X)_\chi\to L\to 0.
$$
And thus $(M_\mu\otimes X)_\chi$ and, consequently, $(U_\psi\otimes X)_\chi$ is nonzero. Moreover, the image of $1\otimes x_{\lam-\mu}$ is nonzero in $L$. Thus, $(F_0U_\psi\otimes X)_\chi$ is nonzero.
\end{proof}

\begin{theorem}\label{tmainth}
Suppose there exist numbers $b_1,\ldots,b_r,c_1,\ldots,c_s\in\C$ such that
$$
\chi(u)-\psi(u) = \sum_{i=1}^r e^{b_i u} - \sum_{i=1}^s e^{(c_i-1)u} .
$$

Then $(U_\psi\otimes S^rV\otimes S^sV^*)_\chi\neq 0$.
\end{theorem} 
\begin{proof}
By Corollary \ref{c10}, there exist weights $\lam^\nn$ and $\mu^\nn$ of $\gl_n(\ck)$ with
$$
\mu^\nn= \lam^\nn - e_1-\ldots-e_r+e_{n-s+1}+\ldots+e_n,
$$
 and central characters $\chi^\nn, \psi^\nn$ corresponding to these weights, such that
$$
\chi = \prod_\cf \chi^\nn, \psi = \prod_\cf \psi^\nn.
$$
 
  Now, for every $n>r+s$ the module $S^rV^\nn\otimes S^s(V^\nn)^*$ has maximal weight $$\lam^\nn- \mu^\nn = e_1+\ldots+e_r-e_{n-s+1}-\ldots - e_n$$
  Thus, by Lemma \ref{l11}, 
  $$
  (U_{\psi^\nn}\otimes S^rV^\nn\otimes S^s(V^\nn)^*)_{\chi^\nn}\neq 0,
  $$
 when $n>s+r$. Moreover,
 $$
  (F_0U_{\psi^\nn}\otimes S^rV^\nn\otimes S^s(V^\nn)^*)_{\chi^\nn}\neq 0,
  $$
  
  We have 
  $$
  (F_0U_{\psi}\otimes S^rV\otimes S^sV^*)_{\chi} = \prod_\cf ( F_0U_{\psi^\nn}\otimes S^rV^\nn\otimes S^s(V^\nn)^*)_{\chi^\nn} \neq 0.
  $$
 And therefore,
 $$
 (U_{\psi}\otimes S^rV\otimes S^sV^*)_\chi = \mathrm{colim}_i \prod_\cf (F_i U_{\psi^\nn}\otimes S^r V^\nn\otimes S^s(V^\nn)^*)_{\chi^\nn} \neq 0.
 $$

\end{proof}

We have now constructed a nonzero object in the category $\hc_{\chi,\psi}$ with
$$
\chi(u)-\psi(u) = \sum_{i=1}^r e^{b_i u} - \sum_{j=1}^s e^{(c_j-1)u}
$$
for any complex numbers $b_i, 1\le i\le r, c_j, 1\le j \le s$, and thus, we have proved Theorem \ref{t8}.

\subsection{Summary}
The following theorem summarizes Corollary \ref{c7}  and Theorem \ref{t8} and is the main result of this paper:

\begin{theorem}
\label{t13}
The category $\hc_{\chi,\psi}$ is nonzero if and only if there exist complex numbers $b_1,\ldots,b_r, c_1, \ldots,c_s$ such that
$$
\chi(u)-\psi(u) = \sum_{i=1}^r e^{b_i u} - \sum_{i=1}^s e^{c_i u} .
$$

\end{theorem}

\begin{note}
We replaced $c_i-1$ with $c_i$ (as compared to Theorem \ref{t8}) to simplify the formula.
\end{note}
\noindent
\textit{Examples.}
Theorem \ref{tmainth} provides us with an example of a nonzero bimodule in $\hc_{\chi,\psi}$ if $\chi,\psi$ satisfy the conditions of Theorem \ref{t13}. We have
$$
0\neq (U_\psi\otimes S^rV\otimes S^sV^*)_\chi\in \hc_{\chi,\psi}.
$$

Other examples can be found in \cite{U}, where we describe a construction of a family of irreducible Harish-Chandra bimodules. These bimodules are a genralization of finite-dimensional bimodules in the classical case and have finite K-type (see Remark \ref{raboutclassichc}).  The corresponding central characters are computed in Section 5 of \cite{U}.

\end{document}